\documentclass[a4paper, 10pt]{article}
\usepackage[T2A]{fontenc}
\usepackage[english]{babel}

\usepackage{amsmath, amsthm}
\usepackage{amssymb}
\usepackage{amsfonts}
\usepackage{srcltx, dsfont, color}

\theoremstyle{plain}
\newtheorem{theorem}{Theorem}[section]
\newtheorem{lemma}{Lemma}[section]

\numberwithin{equation}{section}

\def\tht{\theta}
\def\Om{\Omega}
\def\om{\omega}
\def\e{\varepsilon}

\def\l{\lambda}
\def\p{\partial}
\def\D{\Delta}
\def\a{\alpha}
\def\b{\beta}

\def\d{\delta}

\def\z{\zeta}

\def\vk{\varkappa}

\def\hf{\mathfrak{h}}

\def\la{\langle}
\def\ra{\rangle}

\def\cL{\mathcal{L}}

\DeclareMathOperator{\RE}{Re}

\DeclareMathOperator{\mes}{mes}

\DeclareMathOperator{\dist}{dist}

\DeclareMathOperator*{\esssup}{ess\,sup}

\allowdisplaybreaks

\begin{document}

\title{Operator estimates for non-periodically perforated domains: 
disappearance of cavities}

\author{D.I. Borisov$^{1,2,3}$
}

\date{\empty}

\maketitle

{\small
    \begin{quote}
1) Institute of Mathematics, Ufa Federal Research Center, Russian Academy of Sciences,  Chernyshevsky str. 112, Ufa, Russia, 450008
\\
2) Bashkir State  University, Zaki Validi str. 32, Ufa, Russia, 450076
\\
3) University of Hradec Kr\'alov\'e
62, Rokitansk\'eho, Hradec Kr\'alov\'e 50003, Czech Republic
\\
Emails: borisovdi@yandex.ru
\end{quote}

{\small
\begin{quote}
\noindent{\bf Abstract.}  We consider a boundary value problem for a general second order linear equation in a perforated domain. The perforation is made by small cavities, a minimal distance between the cavities is also small. We impose minimal natural geometric conditions on the shapes of the cavities and no conditions on their distribution in the domain. On the boundaries of the cavities a nonlinear Robin condition is imposed. The sizes of the cavities  and the minimal distance between them are supposed to satisfy a certain simple condition ensuring that under the   homogenization the cavities disappear and we obtain a similar problem in a non-perforated domain.  Our main results state  the convergence of the solution of the perturbed problem to that of the homogenized one in $W_2^1$- and $L_2$-norms uniformly in $L_2$-norm of the right hand side in the equation and provide the estimates for the convergence rates. We also discuss the order sharpness of these estimates.

\medskip

\noindent{\bf Keywords:} perforated domain, non-periodic perforation, operator estimates, convergence rate

\medskip

\noindent{\bf Mathematics Subject Classification:} 35B27, 35B40
 	
\end{quote}
}

\section{Introduction}

During last 20 a new direction devoted to so-called operator estimates appeared in the homogenization theory. Here the boundary value problems  are treated in terms of the spectral theory, that is, the solutions are regarded as actions of the resolvents of the operators corresponding to the considered boundary value problems. The efforts are focused on proving the convergence of the resolvent in an appropriate operator norms, that is, on proving the norm resolvent convergence. The operator estimates are ones for the convergence rates.

Elliptic boundary value problems in domains with a fine perforation distributed along an entire domain are classical in the homogenization theory, see \cite{Sha}, \cite{ZKO}, \cite{MK}, \cite{MK1}, \cite{OIS}, and the references therein, and the classical results stated the convergence in a strong or weak sense in $L_2$ and $W_2^1$ for fixed right hand sides in the equations and boundary conditions. Operator estimates for such problems were recently obtained in few papers \cite{Khra}, \cite{Khra2}, \cite{Past}, \cite{Sus}, \cite{Zhi} for periodic and close-to-periodic perforations.
In \cite{Past}, \cite{Sus}, \cite{Zhi} the case of the Neumann condition was studied in the situation when the sizes of the cavities were of the same order as the distances between them; the perforation was pure periodic. The case of the Dirichlet condition on the boundaries of the cavities
was addressed in \cite{Khra2}; here the sizes of the cavities and the distances between them satisfied certain relation. The cavities were of the same shape but it was allowed to rotate them arbitrary and to locate almost arbitrary within the periodicity cell. In \cite{ChDR}, \cite{Khra},  there was considered the case of a pure periodic perforation by small  balls with the Dirichlet or Neumann \cite{ChDR} or Robin \cite{Khra}  condition on the  boundaries.  In all cited papers various operator estimates were obtained; at the same time, their order sharpness was not established and not discussed.

The case of a non-periodic perforation was addressed in \cite{Post} for a manifold perforated by arbitrary cavities with the Dirichlet or Neumann condition; as the operator, the Laplacian served. The main results were the operator estimates, which were proved under assuming the validity of certain local estimates for $L_2$-norms in terms of $W_2^1$-norms. These local estimates were the key ingredients in the proofs of the operator estimates. As an example it was shown that these local estimates are valid for the perforation by small balls.

There is also a close direction on studying operator estimates for domain perforated along a given manifold, see  \cite{MSB21}, \cite{AA22}, \cite{PRSE16}. Here the perforation was essentially non-periodic: the shapes of the cavities and their distribution were arbitrary and satisfied only minimal geometric assumptions. The main results of the cited papers were the operator estimates, which were shown to be order sharp in many cases.

In two very recent papers \cite{BK}, \cite{JST22}, an elliptic boundary value problem in a perforated domain was considered. The equation was defined by a general second order differential expression  varying coefficients, which was not formally symmetric. The perforation was arbitrary and made along the entire domain. Each cavity had its own shape and the distribution of the cavities in the domain was almost arbitrary. On the boundary of each cavity the Dirichlet or a nonlinear Robin condition was imposed; both boundary conditions were allowed to be present on different boundaries at the same time. The sizes of the cavities and the mutual distances between them satisfied certain conditions, which determined then the form of the homogenized problems. In \cite{BK} these conditions were chosen so that a solution to the perturbed problem tended to zero uniformly in the right hand side in the equation as the perforation became finer. In \cite{JST22} similar conditions ensured the appearance of a strange term in the homogenized problem. In both papers the estimates for the convergence rates were obtained; the order sharpness of these estimates was discussed.

In the present paper we continue studying the general model proposed in \cite{BK}, \cite{JST22} but in a different situation. Namely, now we impose only the nonlinear Robin condition on the boundaries of the cavities and impose conditions ensuring that under the homogenization the cavities disappear and make no contribution to the homogenized problem. In the considered case we succeed to make only minimal and very natural assumptions about the perforation, which are satisfied for a very wide class of non-periodic perforations: we have almost no conditions on the shapes of the cavities and no conditions on their distribution. Our main results are the operator estimates in the considered case. More precisely, we estimate the $W_2^1$- and  $L_2$-norms of the difference of the solutions to the perturbed and homogenized problems uniformly in $L_2$-norm of the right hand side in the equation. The established inequalities provide the estimates for the convergence rates. We show that the estimate for $W_2^1$-norm is order sharp in all dimensions except for $4$, while in the latter dimension it is very close to being order sharp.

\section{Problem and main results}\label{sec2}

Let $\Om$ be  an arbitrary  domain in $\mathds{R}^d$, $d\geqslant 2$, which can be bounded or unbounded. If it has a  boundary, its smoothness is supposed to be $C^2$. Let $M_k^\e\in\Om$, $k\in\mathds{I}^\e$, be a family of points, where $\e$ is a small positive parameter and $\mathds{I}^\e$ is a set of indices, which is at most countable,
 and  $\om_{k,\e}\subset \mathds{R}^d$, $k\in\mathds{I}^\e$, be a family of bounded non-empty domains with $C^1$-boundaries. We define
\begin{equation*}
\om_k^\e:=\big\{x:\, \e^{-1}\eta^{-1}(\e)(x - M_k^\e)\in \om_{k,\e}\big\}, \quad k\in\mathds{I}^\e,\qquad \tht^\e:=\bigcup\limits_{k\in\mathds{I}^\e} \om_k^\e,
\end{equation*}
where $\eta=\eta(\e)$ is a given function such that $0\leqslant \eta(\e)\leqslant 1$. A perforated domain is introduced as   $\Om^\e:=\Om\setminus\tht^\e$.

In the vicinity of the boundaries $\p\om_{k,\e}$ we introduce local variable  $(\tau,s)$, where $\tau$ is the distance measured along the normal vector to $\p\om_{k,\e}$ directed inside $\om_{k,\e}$ and $s$ are some local variables on $\p\om_{k,\e}$. By $B_r(M)$ we denote a ball in $\mathds{R}^d$ of a radius $r$ centered at a point $M$.
The cavities $\om_k^\e$ are supposed to satisfy the following geometric assumption.

\begin{enumerate}
\def\theenumi{A}
\item\label{A1} The points $M_k^\e$ and the domains $\om_{k,\e}$ obey the conditions
\begin{equation}
B_{R_1}(y_{k,\e})\subseteq \om_{k,\e}\subseteq B_{R_2}(0), \qquad
B_{\e R_3}(M_k^\e)\cap B_{\e R_3}(M_j^\e)=\emptyset, \qquad
\dist(M_k^\e,\p\Om)\geqslant R_3\e,
\label{2.2a}
\end{equation}
where $y_{k,\e}$ are some points, $k\ne j$, $k,j\in\mathds{I}^\e$, and $R_1<R_2<R_3$ are some fixed constants independent of $\e$, $\eta$, $k$ and $j$. The sets $B_{R_2}(0)\setminus \om_{k,\e}$ are connected. For each $k\in\mathds{I}^\e$ there exist a local variable $s$ on $\p\om_{k,\e}$ such that the variables $(\tau,s)$ are  well-defined at least on $\{x\in\mathds{R}^d:\, \dist(x,\p\om_{k,\e})\leqslant \tau_0\}\subseteq B_{R_2}(0)$, where $\tau_0$ is a fixed constant independent of $k\in\mathds{I}^\e$ and $\e$. The Jacobians corresponding to passing from variables $x$ to $(\tau,s)$ are separated from zero and bounded from above uniformly in $\e$, $k\in\mathds{M}_R^\e$ and $x$ as $\dist(x,\p\om_{k,\e})\leqslant \tau_0$. The first derivatives of $x$ with respect to $(\tau,s)$ and of $(\tau,s)$ with respect to $x$
are bounded uniformly in $\e$, $k\in\mathds{M}_R^\e$ and $x$ as $\dist(x,\p\om_{k,\e})\leqslant \tau_0$.
\end{enumerate}

By $L_\infty(\Om;\mathds{C}^n)$, $n\geqslant 1$,  we denote the space of vector function with values in $\mathds{C}^n$, each component of which is an element of $L_\infty(\Om)$. The space $L_\infty(\Om;\mathds{C}^n)$ is equipped with the standard norm
$$
\|u\|_{L_\infty(\Om;\mathds{C}^n)}:=\esssup\limits_{x\in\Om} |u(x)|.
$$
By $\mathds{M}_n$ we denote the linear space of all $n\times n$ matrices. The symbol $L_\infty(\Om;\mathds{M}_n)$ stands for the space of matrix functions with values in $\mathds{M}_n$, each entry of which belongs to $L_\infty(\Om)$. Throughout the work we shall also employ similar Lebesgue and Sobolev spaces of vector- and matrix-valued functions; the symbols $\mathds{C}^n$ and $\mathds{M}_n$ will indicate a space, to which the values of these functions belong.

Let $A_{ij}=A_{ij}(x)$, $A_j=A_j(x)$, $A_0=A_0(x)$ be  complex-valued matrix functions  defined on the domain $\Om$ and obeying the conditions
\begin{gather}\label{2.5}
 A_{ij}\in W_\infty^1(\Om;\mathds{M}_n), \qquad A_j,\, A_0\in L_\infty(\Om;\mathds{M}_n),
\\
 \RE \sum\limits_{i,j=1}^{d} \big(A_{ij}(x) z_i, z_j\big)_{\mathds{C}^n}\geqslant c_1 \sum\limits_{i=1}^{n} |z_i|^2,
 \qquad x\in\Om,\quad z_i\in\mathds{C}^n,\label{2.5a}
\end{gather}
where $c_1>0$  is some fixed constant independent of  $x\in\Om$ and $z_i\in\mathds{C}^n$.
By $a^\e=a^\e(x,u)$ we denote a measurable vector function   with values in $\mathds{C}^n$ defined on $\p\tht^\e\times\mathds{C}^n$ and satisfying the following conditions:
\begin{equation}\label{2.6}
|a^\e(x,u_1)-a^\e(x,u_2)|\leqslant \mu(\e) |u_1-u_2|,\qquad a^\e(x,0)=0,
\end{equation}
where  $\mu(\e)$ is some nonnegative function  independent of $x\in\p\tht^\e$ and $u_1, u_2\in\mathds{C}^n$. We assume that for the functions $\eta$ and $\mu$  the convergences hold:
\begin{gather}\label{2.10}
\big(\e\eta(\e)\vk(\e)+\e^{-1}\eta^{n-1}(\e)\big)\mu(\e)\to+0,\qquad \eta(\e)\to+0,  \qquad \e\to+0,
\\
\vk(\e):=|\ln\eta(\e)|\quad\text{as}\quad n=2,\qquad
\vk(\e):=1\quad\text{as}\quad n\geqslant 3.\nonumber
\end{gather}

In this paper we consider the following boundary value problem:
\begin{equation}\label{2.7}
(\cL-\l)u^\e=f\quad\text{in}\quad \Om^\e,\qquad  \frac{\p u^\e}{\p\boldsymbol{\nu}} + a^\e(x,u^\e)=0\quad\text{on}\quad \p\tht^\e,\qquad u^\e=0\quad\text{on}\quad\p\Om.
\end{equation}
Here  $\cL$ and  $\frac{\p\ }{\p\boldsymbol{\nu}}$
are  a differential expression and a conormal derivative:
\begin{equation*}
\cL:=-\sum\limits_{i,j=1}^{d} \frac{\p\ }{\p x_i} A_{ij} \frac{\p\ }{\p x_j}
  + \sum\limits_{j=1}^{n} A_j\frac{\p\ }{\p x_j}  + A_0,\qquad \frac{\p\ }{\p\boldsymbol{\nu}} =\sum\limits_{i,j=1}^{d} \nu_i A_{ij}\frac{\p\ }{\p x_j},
\end{equation*}
$f\in L_2(\Om;\mathds{C}^n)$ is an arbitrary vector function, $\l\in\mathds{C}$ is a fixed constant, $\nu=(\nu_1,\ldots,\nu_d)$ is the unit normal to $\p\tht^\e$ directed inside $\tht^\e$. Our aim is to prove that, under certain conditions, the solution of this problem converges to that of a homogenized problem
\begin{equation}\label{2.11}
(\cL-\l)u_0=f\quad\text{in}\quad \Om,\qquad u_0=0\quad\text{on} \quad \p\Om,
\end{equation}
and to estimate $W_2^1$- and $L_2$-norms of $u_\e-u_0$ uniformly in the $L_2$-norm of the right hand side $f$.

Solutions to problems (\ref{2.7}) and (\ref{2.11}) are  understood in the generalized sense. Namely, a generalized solution to problem (\ref{2.7})
is  a vector function $u\in W_2^1(\Om^\e;\mathds{C}^n)$ with the zero trace on $\p\Om$ such that
\begin{equation}\label{2.18}
\hf^\e(u_\e,v)-\l(u_\e,v)_{L_2(\Om^\e;\mathds{C}^n)} =(f,v)_{L_2(\Om^\e;\mathds{C}^n)}
\end{equation}
for each $v\in W_2^1(\Om^\e;\mathds{C}^n)$ with the zero trace on $\p\Om$, where
\begin{align*}
&\hf^\e(u,v):=\hf(u,v) + (a^\e(\,\cdot\,,u),v)_{L_2(\p\tht^\e;\mathds{C}^n)},
\\
&\hf(u,v):=\sum\limits_{i,j=1}^{d}\left(A_{ij}\frac{\p u}{\p x_j},\frac{\p v}{\p x_i}\right)_{L_2(\Om^\e;\mathds{C}^n)} + \sum\limits_{j=1}^{n} \left(A_j \frac{\p u}{\p x_j}, v\right)_{L_2(\Om^\e;\mathds{C}^n)} + (A_0 u,v)_{L_2(\Om^\e;\mathds{C}^n)}.
\end{align*}
We denote:
\begin{equation*}
\kappa(\e):=0\quad\text{as}\quad n=2,3,\qquad \kappa(\e):=|\ln\eta(\e)|^\frac{1}{2}\quad\text{as}\quad n=4,
\qquad \kappa(\e):=1\quad\text{as}\quad n\geqslant 5.
\end{equation*}

Our main result is as follows.

\begin{theorem}\label{th1}
Let Assumption~\ref{A1} and (\ref{2.10}) be satisfied.   Then there exists a fixed $\l_0\in\mathds{R}$ independent of $\e$ such that as $\RE\l\leqslant \l_0$, problems (\ref{2.7}), (\ref{2.11}) are uniquely solvable for each $f\in L_2(\Om;\mathds{C}^n)$ and the solutions  satisfy the estimate:
\begin{equation}\label{2.16}
\|u_\e-u_0\|_{W_2^1(\Om^\e;\mathds{C}^n)}  \leqslant  C(\l) \big( (\e^2\eta^2\kappa+\eta^{\frac{n}{2}}\vk^{\frac{1}{2}} + \e^{-1}\eta^{n-1})\mu+\e\eta\vk^\frac{1}{2} +  \eta^\frac{n}{2}\big)
\|f\|_{L_2(\Om;\mathds{C}^n)},
\end{equation}
where $C$ is some constant independent of $\e$ and $f$. This estimate is order sharp for $n\ne 4$, while for $n=4$ all terms  in the right hand except for $\e^2\eta^2\kappa\mu$ are order sharp.
If, in addition, $A_j\in W_\infty^1(\Om;\mathds{C}^n)$,
then
\begin{equation}
\label{2.17}
\begin{aligned}
\|u_\e-u_0\|_{L_2(\Om^\e;\mathds{C}^n)}\leqslant C \big(& (\e^2\eta^2\kappa+\eta^{\frac{n}{2}}\vk^{\frac{1}{2}} + \e^{-1}\eta^{n-1}) \mu +\e^2\eta^2\vk +  \eta^n \big)  \|f\|_{L_2(\Om;\mathds{C}^n)}
 \\
 &+\big(\e\eta\vk^\frac{1}{2}+\eta^\frac{n}{2}\big) \|f\|_{L_2(\tht^\e;\mathds{C}^n)},
\end{aligned}
\end{equation}
where $C$ is some constant independent of $\e$ and $f$.
\end{theorem}

Let us briefly discuss the problem and main results. We consider a general non-periodic perforation under minimal conditions both for the distribution of the cavities and for their shapes introduced  in Assumption~\ref{A1}. The first two-sided relation in (\ref{2.2a}) says that all cavities are approximately of the same size (but not the shapes!), while two other relations in (\ref{2.2a}) just mean that all cavities are mutually disjoint and do not intersect the boundary of the domain $\Om$. The minimal distance between the cavities is $2R_3\e$ and at the same time, there is no apriori upper bound for the mutual distances. In particular, our model covers also the situations, when some of the distances between some cavities are much larger than $2R_3\e$, for instance, these distances can be finite. The total number of the cavities can be also arbitrary including the case when we deal with finitely many cavities separated by finite distances.

The regularity of the boundaries of the cavities postulated in  Assumption~\ref{A1} is rather natural and not very restrictive. The stated existence of local variables  $(\tau,s)$ is an implicit condition for the uniform (in $k$) regularity of the shapes of the cavities and excludes, for instance, the situation, when the boundaries of the cavities have increasing oscillations on some sequence of the values of $k$. We stress that these conditions do not mean that all cavities are of the same shapes.
Our conditions on the perforation are very weak and natural and we in fact deal with a very large class of non-periodic perforations.

On the boundaries of all cavities we impose the nonlinear Robin condition, see (\ref{2.7}). The growth of the nonlinearity is controlled by conditions (\ref{2.6}). They mean that the nonlinearity is allowed to have at most a linear growth in $u$, while the parameter $\mu$ describes the strength of the nonlinearity. The sizes of all cavities are of order $O(\e\eta)$ and the parameters $\eta$ and $\mu$ are to satisfy convergences (\ref{2.10}). This is the main assumption ensuring that under the homogenization the cavities disappear and make no contribution to the limiting (homogenized) problem. We also stress that the differential expression $\cL$ involved in the equations in these problems is of a general form with varying complex-valued matrix coefficients and this expression is not formally symmetric. It is also important to say that since the coefficients are matrix functions, we in fact deal with a system of scalar equations.

Our main theorem states that under the above discussed conditions, the solution of  problem (\ref{2.7}) converges to that of the homogenized problem uniformly in $L_2(\Om)$-norm of the right hand side in the equation. In the case of the linear Robin condition this means that the linear operator associated with the perturbed problem converges to that associated with the homogenized problem in the norm resolvent sense. Our inequality (\ref{2.16}) provides an estimate for the convergence rate, which even turns out to be order sharp in the dimension $n\ne 4$, while for $n=4$ it is almost order sharp. Namely, as $n=4$, all terms except for $\e^2\eta^2\kappa\mu=\e^2\eta^2\mu|\ln\eta|$ are order sharp. The latter term is also close to being sharp, namely, in Section~\ref{sec:sharp}, while checking the sharpness by providing an appropriate example, we in fact show that the convergence rate can not be smaller than $\e^2\eta^2$ for all dimensions.

Our second estimate (\ref{2.17}) provides the convergence rate in the case, when the difference between the solutions of the perturbed and homogenized problems is measured in $L_2$-norm. Here we see that the coefficient at $\mu$ is the same as in (\ref{2.16}), while the coefficients without $\mu$ at $\|f\|_{L_2(\Om;\mathds{C}^n)}$ have the smallness order twice more than the similar coefficient  in (\ref{2.16}). This is due to the fact that in (\ref{2.17}) we estimate a weaker norm of the difference $u_\e-u_0$ than  in (\ref{2.16}). There is also an additional term in (\ref{2.17}) depending on $\|f\|_{L_2(\tht^\e;\mathds{C}^n)}$. The coefficient at this norm is the same as in (\ref{2.16}) at $\|f\|_{L_2(\Om;\mathds{C}^n)}$. At the same time, the values of the function $f$ in $\tht^\e$ make no influence on the solution of the perturbed problem. In view of this fact, we could apriori suppose that $f$ vanishes on $\tht^\e$ for the considered values of $\e$ but in this case the right hand side becomes $\e$-dependent and the same concerns $u_0$. Nevertheless, in this situation Theorem~\ref{th1} still makes sense since it describes the error we make by replacing the perforated domain by the non-perforated one and this error is estimated uniformly in the right hand side $f$. If we consider the case of some fixed $\e$-independent function $f$ then it is reasonable to suppose that it is non-zero on $\tht^\e$. In such situation the presence of the additional term in (\ref{2.17}) concerns estimating the contribution of the restriction of $f$ on $\tht^\e$ into the function $u_0$ and not to the difference $u_\e-u_0$. The order sharpness of inequality (\ref{2.17}) remained an open question: our examples in Section~\ref{sec:sharp} turned out to be inappropriate for such $L_2$-estimate.

\section{Operator estimates}

In this section we prove Theorem~\ref{th1}. We begin with auxiliary lemmata proved
in \cite{BK}, \cite{JST22}.  The first of them is Lemma~3.6 from \cite{BK}.

\begin{lemma}\label{lm3.1}
Under Assumption~\ref{A1}
for all $k\in\mathds{I}^\e$ and all $u\in W_2^1(B_{\e R_3}(M_k^\e)\setminus\om_k^\e)$ the estimate
\begin{equation*}
\|u\|_{L_2(\p\om_k^\e)}^2 \leqslant  C\Big(\e\eta \vk\|\nabla u\|_{L_2(B_{\e R_3}(M_k^\e)\setminus\om_k^\e)}^2 +\e^{-1}\eta^{n-1}
\|u\|_{L_2(B_{\e R_3}(M_k^\e)\setminus B_{\e R_2}(M_k^\e))}^2\Big)
\end{equation*}
holds, where  $C$ is a constant independent of the parameters $k$, $\e$, $\eta$ and the function $u$.
\end{lemma}

Taking $u\equiv 1$ in the above lemma, we immediately find that the $(n-1)$-dimensional measures of $\p\om_{k,\e}$ are uniformly bounded:
\begin{equation}\label{3.9}
\mes_{n-1} \p\om_{k,\e}\leqslant C\e^{n-1}\eta^{n-1},
\end{equation}
where $C$ is a constant independent of $k$, $\e$ and $\eta$.

The next statement is Lemma~3.2 from \cite{JST22}.  Although in \cite{JST22} it was assumed that the boundaries of the domains  $\om_{k,\e}$ had the smoothness $C^2$, this fact was not used in the proof of Lemma~3.2 and this is why it is valid also under our Assumption~\ref{A1}.

\begin{lemma}\label{lm3.6}
Under Assumption~\ref{A1}  for all $k\in\mathds{I}^\e$ and all $u\in W_2^1(B_{\e R_3}(M_k^\e)\setminus\om_k^\e)$ the estimate
\begin{equation*}
\|u\|_{L_2(B_{\e\eta R_3}(M_k^\e)\setminus\om_k^\e)}^2 \leqslant C  \Big(\e^2\eta^2 \vk\|\nabla u\|_{L_2(B_{\e R_3}(M_k^\e)\setminus\om_k^\e)}^2 + \eta^n
\|u\|_{L_2(B_{\e R_3}(M_k^\e)\setminus\om_k^\e)}^2\Big)
\end{equation*}
holds   with a constant $C$ independent of $k$, $\e$ and $u$.
\end{lemma}

The next lemma is \cite[Lm. 3.5]{BK}.

\begin{lemma}\label{lm3.3}
Under  Assumption~\ref{A1}
for all $k\in\mathds{I}^\e$ and all $u\in W_2^1(B_{\e\eta R_3}(M_k^\e)\setminus\om_k^\e)$ obeying the identity
\begin{equation}\label{3.1}
\int\limits_{B_{\e\eta  R_3}(M_k^\e)\setminus\om_k^\e} u(x)\,dx=0
\end{equation}
the estimate
\begin{align}
&\|u\|_{L_2(\p\om_k^\e)}^2\leqslant C\e\eta \|\nabla u\|_{L_2(B_{\e\eta R_3}(M_k^\e)\setminus\om_k^\e)}^2,\label{3.2c}
\end{align}
holds, where $C$ is a constant independent of the parameters $k$, $\e$, $\eta$ and the function $u$.
\end{lemma}

\begin{lemma}\label{lm3.4}
For all $u\in W_2^2(B_{\e R_3}(M_k^\e))$ the estimate
\begin{equation*}
\|u\|_{L_2(\p\om_k^\e)}^2\leqslant C (\e^{-1}\eta^{n-1}+\kappa^2\e^3\eta^3) \|u\|_{W_2^2(B_{R_3\e}(M_k^\e))}^2
\end{equation*}
holds true, where  $C$ is a constant independent of the parameters $k$, $\e$, $\eta$ and the function $u$.
\end{lemma}

\begin{proof} Let $u\in W_2^2(B_{\e R_3}(M_k^\e))$ be a given function.
We first consider the case $n=2,3$. We introduce one more function $\tilde{u}(\xi):=u(M_k^\e+\e\xi)$. It belongs to $W_2^2(B_{R_3}(0))$ and hence, due to the embedding of the latter space into $C(\overline{B_{R_3}(0)})$,
\begin{equation*}
\|\tilde{u}\|_{C(\overline{B_{R_3}(0)})}\leqslant C \|\tilde{u}\|_{W_2^2(B_{R_3}(0))};
\end{equation*}
throughout the proof the symbol $C$ stands for various inessential constants independent of $k$, $\e$, $\eta$, $u$. Returning back to the function $u$, we immediately obtain one more estimate
\begin{equation*}
\|u\|_{C(\overline{B_{R_3\e}(M_k^\e)})}^2\leqslant C\e^{-n} \|u\|_{W_2^2(B_{R_3\e}(M_k^\e))}^2.
\end{equation*}
Hence, in view of  (\ref{3.9}),
\begin{equation*}
\|u\|_{L_2(\p\om_k^\e)}^2\leqslant C \e^{-1}\eta^{n-1} \|u\|_{W_2^2(B_{R_3\e}(M_k^\e))}^2,
\end{equation*}
and this proves the desired estimate for $n=2,3$.

We proceed to the case $n\geqslant 4$. We denote
\begin{equation*}
\la u\ra:=\frac{1}{\mes B_{\e\eta R_3}} \int\limits_{B_{\e\eta R_3}} u\,dx,\qquad u_\bot:=u-\la u\ra,\qquad u^\bot(\xi):=u_\bot(M_k^\e+\e\eta\xi).
\end{equation*}
We obviously have:
\begin{equation}\label{3.8}
\|u\|_{L_2(\p\om_k^\e)}^2\leqslant C (\e\eta)^{-n-1} \Bigg|\int\limits_{B_{\e\eta R_3}} u\,dx\Bigg|^2
+ C \|u_\bot\|_{L_2(\p\om_k^\e)}^2.
\end{equation}
 The function $u^\bot$ belongs to $W_2^1(B_{R_3}(0))$ and satisfies the identity
\begin{equation*}
\int\limits_{B_{R_3}(0)} u^\bot\,d\xi=0.
\end{equation*}
Then, by \cite[Lm. 3.4]{BK},
\begin{equation*}
\|u^\bot\|_{L_2(\p\om_{k,\e})} \leqslant C \|\nabla u^\bot\|_{L_2(B_{R_3}(0))}.
\end{equation*}
Returning back to the function $u_\bot$ and using Lemma~\ref{lm3.6}, we obtain:
\begin{equation}\label{3.4}
\|u_\bot\|_{L_2(\p\om_k^\e)}^2\leqslant C \e\eta \|\nabla u\|_{L_2(B_{\e R_3}(M_k^\e))}^2\leqslant
C \big(\e^3\eta^3  + \e\eta^{n+1}\big)\|u\|_{W_2^2(B_{\e R_3}(M_k^\e))}^2.
\end{equation}

We introduce an auxiliary function:
\begin{equation*}
X(x):=\left\{
\begin{aligned}
& \frac{|x|^2-R_3^2\e^2\eta^2}{2n} + \frac{R_3^2\e^2\eta^n}{n(2-n)} (\eta^{-n+2}-1)  && \text{in}
\quad B_{R_3\e\eta}(0),
\\
&  \frac{R_3^2\e^2\eta^n}{n(2-n)}   \left(\left(\frac{R_3\e}{|x|}\right)^{n-2}-1\right) &&
 \text{in}
\quad B_{R_\e \e}(0)\setminus B_{R_3\e\eta}(0).
\end{aligned}\right.
\end{equation*}
This function solves the following boundary value problem:
\begin{equation*}
\D X=1\quad\text{in}\quad B_{\e\eta R_3}(0),\qquad \D X=0\quad\text{in}\quad B_{R_3\e}(0)\setminus B_{R_3\e\eta}(0),\qquad X=0\quad\text{on}\quad \p B_{\e R_3}(0).
\end{equation*}
Using this problem, we integrate by parts as follows:
\begin{align*}
\int\limits_{B_{R_3\e\eta}(M_k^\e)} u(x)\,dx=&\int\limits_{B_{R_3\e}(M_k^\e)} u(x)\D X(x-M_k^\e)\,dx
\\
=& \int\limits_{\p B_{R_3\e}(M_k^\e)} u(x)\frac{\p X}{\p|x|}(x-M_k^\e)\,ds + \int\limits_{B_{R_3\e}(M_k^\e)} X(x-M_k^\e)\D u(x)\,dx.
\end{align*}
The explicit formula for $X$ and Cauchy-Schwarz inequality then yield:
\begin{equation}\label{3.5}
\begin{aligned}
\Bigg|\int\limits_{B_{R_3\e\eta}(M_k^\e)} u\,dx\Bigg|^2 \leqslant C\Big(\e^{n+1}\eta^{2n} \|u\|_{L_2(\p B_{R_3\e}(M_k^\e))}^2 + \|X\|_{L_2(B_{R_3\e}(0))}^2 \|u\|_{W_2^2(M_k^\e))}^2\Big).
\end{aligned}
\end{equation}

The function $\tilde{u}(\xi):=u(M_k^\e+\e\xi)$ belongs to $W_2^1(B_{R_3}(0))$  and satisfies the estimate
\begin{equation*}
\|\tilde{u}\|_{L_2(\p B_{R_3}(0))}^2\leqslant C \|\tilde{u}\|_{W_2^1(B_{R_3}(0))}^2.
\end{equation*}
Returning back to the function $u$, we then find:
\begin{equation}\label{3.7}
\|u\|_{L_2(\p B_{\e R_3}(0))}^2\leqslant C \e^{-1} \|u\|_{W_2^1(B_{\e R_3}(0))}^2.
\end{equation}
Using the explicit formula for $X$ once again, we calculate the $L_2(B_{R_3\e}(M_k^\e))$-norm of $X$ and we get:
\begin{equation*}
\|X\|_{L_2(B_{R_3\e}(0))}^2 \leqslant C (\e\eta)^{n+4}, \quad n\geqslant 5, \qquad
\|X\|_{L_2(B_{R_3\e}(0))}^2 \leqslant C (\e\eta)^8 |\ln\eta|, \quad n=4.
\end{equation*}
We substitute this estimate and (\ref{3.7}) into (\ref{3.5}) and then the resulting inequality is combined with (\ref{3.4}), (\ref{3.8}). This gives the desired inequality and completes the proof.
\end{proof}

Let us prove the unique solvability of problems (\ref{2.7}), (\ref{2.11}). In order to do this, we follow   general results from the theory of monotone operators, see \cite[Ch. 1, Sect. 1.2$^0$]{Dub},
\cite[Ch. V\!I, Sect. 18.4]{Vain}. According to these results,  in  our case the unique solvability is ensured by the following conditions:
\begin{enumerate}
\item\label{1} For all $u,v,w\in \mathfrak{V}$ the function $t\mapsto\hf^\e(u+t  v,w)-\l(u+t v,w)_{L_2(\Om^\e;\mathds{C}^n)}$ is continuous;
\item\label{2} For all $u,v\in \mathfrak{V}$ the inequality $\RE\big(\hf^\e(u-v,u-v)-\l\|u-v\|_{L_2(\Om^\e;\mathds{C}^n)}^2\big)>0$ holds;
\item\label{3} The convergence is valid:
\begin{equation*}
\frac{\RE\big(\hf^\e(u,u)-\l\|u\|_{L_2(\Om^\e;\mathds{C}^n)}^2\big)} {\|u\|_{W_2^1(\Om^\e;\mathds{C}^n)}}\to+\infty\quad\text{as}\quad \|u\|_{W_2^1(\Om^\e;\mathds{C}^n)}\to+\infty.
\end{equation*}
\end{enumerate}
We are going to check all these three conditions for the form $\hf^\e$
as well as the same conditions for the form
\begin{equation*}
\hf^0(u,v):=\sum\limits_{i,j=1}^{d}\left(A_{ij}\frac{\p u}{\p x_j},\frac{\p v}{\p x_i}\right)_{L_2(\Om;\mathds{C}^n)} + \sum\limits_{j=1}^{n} \left(A_j \frac{\p u}{\p x_j}, v\right)_{L_2(\Om;\mathds{C}^n)} + (A_0 u,v)_{L_2(\Om;\mathds{C}^n)}.
\end{equation*}
This will prove the unique solvability of problems (\ref{2.7}), (\ref{2.11}).

Condition \ref{1} obviously holds true for both forms $\hf^\e$ and $\hf^0$. By conditions (\ref{2.5}), (\ref{2.5a}) and the Cauchy-Schwartz inequality we immediately get:
\begin{align*}
&\RE \hf(u,u)\geqslant (c_1-\d) \|\nabla u\|_{L_2(\Om^\e;\mathds{C}^n)}^2
-C\d^{-1} \|u\|_{L_2(\Om^\e;\mathds{C}^n)}^2 &&\hspace{-1.5 true cm} \text{for all}\quad u\in W_2^1(\Om^\e;\mathds{C}^n),
\\
&\RE \hf^0(u,u)\geqslant (c_1-\d) \|\nabla u\|_{L_2(\Om;\mathds{C}^n)}^2
-C\d^{-1} \|u\|_{L_2(\Om;\mathds{C}^n)}^2 &&\hspace{-1.5 true cm} \text{for all}\quad u\in W_2^1(\Om;\mathds{C}^n),
\end{align*}
where $\d>0$ is arbitrary and fixed, while $C>0$ is some constant independent of $\d$ and $u$. Choosing then $\d:=c_1/2$ and $\l_0:=1+2C/c_1$, we obtain the estimate
\begin{equation}\label{4.14}
\RE\big(\hf^0(u,u)-\l \|u\|_{L_2(\Om;\mathds{C}^n)}^2\big)\geqslant \frac{c_1}{2}\|\nabla u\|_{L_2(\Om;\mathds{C}^n)}^2 + \|u\|_{L_2(\Om;\mathds{C}^n)}^2
\end{equation}
for $\l\in\mathds{C}$ such that $\RE\l\leqslant \l_0$. This estimate implies the validity of Conditions~\ref{2},~\ref{3} for the form $\hf^0$.

It follows from (\ref{2.6}) that
\begin{align*}
&\big|\big(a^\e(\,\cdot\,,u)-a^\e(\,\cdot\,,v), u-v\big)_{L_2(\p\tht^\e;\mathds{C}^n)}\big| \leqslant \mu \|u-v\|_{L_2(\p\tht^\e;\mathds{C}^n)}^2=\mu \sum\limits_{k\in\mathds{I}^\e} \|u-v\|_{L_2(\p\om_k^\e;\mathds{C}^n)}^2,
\\
&\big|\big(a^\e(\,\cdot\,,u),u\big)_{L_2(\p\tht^\e;\mathds{C}^n)}\big| \leqslant \mu \|u\|_{L_2(\p\tht^\e;\mathds{C}^n)}^2= \mu \sum\limits_{k\in\mathds{I}^\e} \|u\|_{L_2(\p\om_k^\e;\mathds{C}^n)}^2.
\end{align*}
Applying Lemma~\ref{lm3.1}, we find:
\begin{align*}
&\big|\big(a^\e(\,\cdot\,,u)-a^\e(\,\cdot\,,v), u-v\big)_{L_2(\p\tht^\e;\mathds{C}^n)}\big| \leqslant C\big(\e\eta\vk+\e^{-1}\eta^{n-1}\big)\mu \|u-v\|_{L_2(\p\Om^\e;\mathds{C}^n)}^2,
\\
&\big|\big(a^\e(\,\cdot\,,u),u\big)_{L_2(\p\tht^\e;\mathds{C}^n)}\big| \leqslant C\big(\e\eta\vk+\e^{-1}\eta^{n-1}\big)\mu \|u\|_{L_2(\p\Om^\e;\mathds{C}^n)}^2,
\end{align*}
where $C$ is some constant independent of $\e$, $\mu$ and $u$.
Taking into consideration the first convergence in (\ref{2.10}), by the above estimates and (\ref{4.14}) with $\d=3c_1/4$ we get:
\begin{align}
&\RE\big(\hf^\e(u-v,u-v)-\l\|u-v\|_{L_2(\Om^\e;\mathds{C}^n)}^2\big) \geqslant \frac{c_1}{2} \|\nabla(u-v)\|_{L_2(\Om^\e;\mathds{C}^n)}^2+
\|u-v\|_{L_2(\Om^\e;\mathds{C}^n)}^2,\nonumber
\\
&\RE\big(\hf^\e(u,u)-\l\|u\|_{L_2(\Om^\e;\mathds{C}^n)}^2\big) \geqslant \frac{c_1}{2} \|\nabla u\|_{L_2(\Om^\e;\mathds{C}^n)}^2+
\|u\|_{L_2(\Om^\e;\mathds{C}^n)}^2,\label{3.6}
\end{align}
for all $u,v\in W_2^1(\Om^\e;\mathds{C}^n)$ and $\l\in\mathds{C}$ such that $\RE\l\leqslant \l_0$, where $\l_0$ is the same as above. These two estimates imply Conditions~\ref{2},~\ref{3} for the form $\hf^\e$. Hence, the problems (\ref{2.7}), (\ref{2.11}) are uniquely solvable for all $f\in L_2(\Om;\mathds{C}^n)$. Estimate (\ref{3.6}) also yields
\begin{equation}\label{2.12}
\|u_\e\|_{W_2^1(\Om^\e;\mathds{C}^n)}\leqslant C \|f\|_{L_2(\Om^\e;\mathds{C}^n)},
\end{equation}
where $C$ is some constant independent of $\e$ and $f$. By (\ref{4.14}) we get a similar estimate for $\|u_0\|_{W_2^1(\Om;\mathds{C}^n)}$. Using then standard smoothness improving theorems, we obtain:
\begin{equation}
\|u_0\|_{W_2^2(\Om;\mathds{C}^n)}\leqslant C(\l)\|f\|_{L_2(\Om;\mathds{C}^n)},\label{2.15}
\end{equation}
with some constant $C$ independent of $f$.

We proceed to proving inequality (\ref{2.16}). Given an arbitrary $f\in L_2(\Om)$, we let $v_\e:=u_\e- u_0$, where
$u_\e$ and $u_0$ are the solutions of problems (\ref{2.7}), (\ref{2.11}). The function $v_\e$ is an element of $W_2^1(\Om^\e;\mathds{C}^n)$ and has the zero trace on $\p\Om$. Integral identity (\ref{2.18}) with $v_\e$ as the test function reads
\begin{equation}\label{4.12}
\hf^\e(u_\e,v_\e)-\l(u_\e,v_\e)_{L_2(\Om^\e;\mathds{C}^n)} =(f,v_\e)_{L_2(\Om^\e;\mathds{C}^n)}.
\end{equation}
Since $u_0$ belongs to $W_2^2(\Om;\mathds{C}^n)$, we can
multiply the equation in (\ref{2.11}) by $v_\e$
and integrate by parts over $\Om^\e$. This gives:
\begin{equation}\label{4.13}
\hf(u_0,v_\e)-\left(\frac{\p u_0}{\p\boldsymbol{\nu}},v_\e\right)_{L_2(\p\tht^\e;\mathds{C}^n)} -\l(u_0,v_\e)_{L_2(\Om^\e;\mathds{C}^n)} =(f,v_\e)_{L_2(\Om^\e;\mathds{C}^n)}.
\end{equation}
Calculating the difference of this identity with (\ref{4.12}), we obtain:
\begin{equation}\label{4.1}
\hf^\e(v_\e,v_\e)-\l\|v_\e\|_{L_2(\Om^\e;\mathds{C}^n)}^2=g_\e,
\qquad g_\e:=\big(a^\e(\,\cdot\,,v_\e) -a^\e(\,\cdot\,,u_\e),v_\e\big)_{L_2(\p\tht^\e;\mathds{C}^n)}
-\left(\frac{\p u_0}{\p\boldsymbol{\nu}},v_\e\right)_{L_2(\p\tht^\e;\mathds{C}^n)}.
\end{equation}
By inequality (\ref{3.6}) we have
\begin{equation}\label{4.2}
\RE \Big(\hf^\e(v_\e,v_\e)-\l\|v_\e\|_{L_2(\Om^\e;\mathds{C}^n)}^2\Big)\geqslant
C\|v_\e\|_{W_2^1(\Om^\e;\mathds{C}^n)}^2.
\end{equation}
Hereinafter, by $C$ we denote various inessential constants independent of $f$, $u_0$, $u_\e$, $v_\e$, $\e$, $k$ and spatial variables but, in general, depending on $\l$.

Our next step is to estimate the real part of the right hand side in (\ref{4.1}). It follows from the second inequality in (\ref{2.6}) that
\begin{equation*}
\big|a^\e(\,\cdot\,,v_\e) -a^\e(\,\cdot\,,u_\e)\big|\leqslant \mu(\e)|v_\e-u_\e|=\mu(\e)|u_0|.
\end{equation*}
Hence, by Lemmata~\ref{lm3.1},~\ref{lm3.4} and estimate (\ref{2.15}),
\begin{equation}\label{4.16}
\begin{aligned}
\big|\big(a^\e(\,\cdot\,,v_\e) -a^\e(\,\cdot\,,u_\e),v_\e\big)_{L_2(\p\tht^\e;\mathds{C}^n)}\big| \leqslant & C\mu\sum\limits_{k\in\mathds{I}^\e} \|u_0\|_{L_2(\p\om_k^\e;\mathds{C}^n)} \|v_\e\|_{L_2(\p\om_k^\e;\mathds{C}^n)}
\\
\leqslant & C \big(\e\eta\vk+\e^{-1}\eta^{n-1}\big)^{\frac{1}{2}} \big(\kappa+\e^{-1}\eta^{n-1}\big)^{\frac{1}{2}}
\mu
\\
&\cdot\sum\limits_{k\in\mathds{I}^\e} \|u_0\|_{W_2^1(B_{\e R_3}(M_k^\e)\setminus\om_k^\e;\mathds{C}^n)} \|v_\e\|_{W_2^1(B_{\e R_3}(M_k^\e)\setminus\om_k^\e;\mathds{C}^n)}
\\
\leqslant & C \big(\e^2\eta^2\kappa+\eta^{\frac{n}{2}}\vk^{\frac{1}{2}} + \e^{-1}\eta^{n-1}
\big)\mu \|u_0\|_{W_2^1(\Om^\e;\mathds{C}^n)} \|v_\e\|_{W_2^1(\Om^\e;\mathds{C}^n)}
\\
\leqslant & C  \big(\e^2\eta^2\kappa+\eta^{\frac{n}{2}}\vk^{\frac{1}{2}} + \e^{-1}\eta^{n-1}
\big) \mu \|f\|_{L_2(\Om^\e;\mathds{C}^n)} \|v_\e\|_{W_2^1(\Om^\e;\mathds{C}^n)}.
\end{aligned}
\end{equation}

For each $k\in\mathds{I}^\e$ we define
\begin{equation*}
\la v_\e\ra_k:=\frac{1}{\mes B_{\e\eta R_3}(M_k^\e)\setminus\om_k^\e} \int\limits_{B_{\e \eta R_3}(M_k^\e)\setminus\om_k^\e} v_\e\,dx, \qquad
v_{\e,k}:=v_\e-\la v_\e\ra_k.
\end{equation*}
Then we can rewrite the second term in $g_\e$ as
\begin{equation}\label{4.17}
\begin{aligned}
\left(\frac{\p u_0}{\p\boldsymbol{\nu}},v_\e\right)_{L_2(\p\tht^\e;\mathds{C}^n)} = &
\sum\limits_{k\in\mathds{I}^\e} \left(\frac{\p u_0}{\p\boldsymbol{\nu}},v_\e\right)_{L_2(\p\om_k^\e;\mathds{C}^n)}
\\
= &\sum\limits_{k\in\mathds{I}^\e} \Bigg(\int\limits_{\p\om_k^\e} \frac{\p u_0}{\p\boldsymbol{\nu}}\,ds,
\la v_\e\ra_k\Bigg)_{\mathds{C}^n} + \sum\limits_{k\in\mathds{I}^\e} \left(\frac{\p u_0}{\p\boldsymbol{\nu}},v_{\e,k}\right)_{L_2(\p\om_k^\e;\mathds{C}^n)}.
\end{aligned}
\end{equation}
The functions $v_{\e,k}$ obey condition (\ref{3.1}) and this is why by estimate (\ref{3.2c}) we have:
\begin{align*}
\|v_{\e,k}\|_{L_2(\p\om_k^\e;\mathds{C}^n)}^2 \leqslant & C
\e\eta
\|\nabla v_\e\|_{L_2(B_{\e\eta R_3}(M_k^\e)\setminus\om_k^\e)}^2.
\end{align*}
Hence, by Lemmata~\ref{lm3.1},~\ref{lm3.3} and inequality (\ref{2.15}),
\begin{equation}\label{4.18}
\begin{aligned}
\Bigg| \sum\limits_{k\in\mathds{I}^\e} \left(\frac{\p u_0}{\p\boldsymbol{\nu}},v_{\e,k}\right)_{L_2(\p\om_k^\e;\mathds{C}^n)}
\Bigg| \leqslant &C   \big(\e\eta\vk+\e^{-1}\eta^{n-1}\big)^\frac{1}{2} (\e\eta)^\frac{1}{2} \|u_0\|_{W_2^2(\Om^\e;\mathds{C}^n)}
\|v_\e\|_{W_2^1(\Om^\e;\mathds{C}^n)}
\\
\leqslant &C   \big(\e\eta\vk^\frac{1}{2}+\eta^\frac{n}{2}\big) \|f\|_{L_2(\Om;\mathds{C}^n)}
\|v_\e\|_{W_2^1(\Om^\e;\mathds{C}^n)}.
\end{aligned}
\end{equation}
Since the function $u_0$ belongs to $W_2^2(\Om;\mathds{C}^n)$, we can integrate by parts as follows:
\begin{equation}\label{4.25}
\int\limits_{\p\om_k^\e} \frac{\p u_0}{\p\boldsymbol{\nu}}\,ds = -\int\limits_{\om_k^\e} \sum\limits_{i,j=1}^{d} \frac{\p\ }{\p x_i} A_{ij} \frac{\p u_0}{\p x_j}\,dx.
\end{equation}
Therefore,
\begin{equation}\label{3.3}
\Bigg|\int\limits_{\p\om_k^\e} \frac{\p u_0}{\p\boldsymbol{\nu}}\,ds
\Bigg|\leqslant C \e^{\frac{n}{2}}\eta^{\frac{n}{2}} \|u_0\|_{W_2^2(\om_k^\e;\mathds{C}^n)}.
\end{equation}
We also observe that by the Cauchy-Schwartz inequality we have
\begin{equation*}
|\la v_\e\ra_k|\leqslant C \e^{-\frac{n}{2}}\eta^{-\frac{n}{2}} \|v_\e\|_{L_2(B_{\e\eta R_3}(M_k^\e)\setminus\om_k^\e)}.
\end{equation*}
Using this inequality, (\ref{3.3}), (\ref{2.15}), (\ref{3.9}) and Lemma~\ref{lm3.6}, we estimate the first term in the right hand side of (\ref{4.17}):
\begin{equation}\label{4.19}
\begin{aligned}
\Bigg|\sum\limits_{k\in\mathds{I}^\e} \Bigg(\int\limits_{\p\om_k^\e} \frac{\p u_0}{\p\boldsymbol{\nu}}\,ds,
\la v_\e\ra_k\Bigg)_{\mathds{C}^n}\Bigg|\leqslant &C
\|u_0\|_{W_2^2(\Om^\e;\mathds{C}^n)} \|v_\e\|_{L_2(B_{\e\eta R_3}(M_k^\e)\setminus\om_k^\e;\mathds{C}^n)}
\\
\leqslant & C \big(\e\eta\vk^\frac{1}{2}+\eta^{\frac{n}{2}}\big)
\|f\|_{L_2(\Om^\e;\mathds{C}^n)} \|v_\e\|_{W_2^1(\Om^\e;\mathds{C}^n)}.
\end{aligned}
\end{equation}
This estimate and (\ref{4.18}), (\ref{4.17}), (\ref{4.16}) yield:
\begin{equation*}
|g_\e|\leqslant  C \big(\big(\e^2\eta^2\kappa+\eta^{\frac{n}{2}}\vk^{\frac{1}{2}} + \e^{-1}\eta^{n-1}\big)\mu+\e\eta\vk^\frac{1}{2}+  \eta^\frac{n}{2}\big)  \|f\|_{L_2(\Om^\e;\mathds{C}^n)} \|v_\e\|_{W_2^1(\Om^\e;\mathds{C}^n)}
\end{equation*}
Having this inequality in mind, we take the real part of identity (\ref{4.1}) and in view of (\ref{4.2}) we obtain
\begin{equation}\label{4.5}
\|v_\e\|_{W_2^1(\Om^\e;\mathds{C}^n)}\leqslant C\big(\big(\e^2\eta^2\kappa+\eta^{\frac{n}{2}}\vk^{\frac{1}{2}} + \e^{-1}\eta^{n-1}\big)\mu+\e\eta\vk^\frac{1}{2}+  \eta^\frac{n}{2}\big)\|f\|_{L_2(\Om;\mathds{C}^n)},
\end{equation}
which proves inequality (\ref{2.16}).

We proceed to proving (\ref{2.17}). We use an approach proposed recently in \cite{PMA22-2} as a modification of the technique developed in \cite{Pas1}, \cite{Pas2}, \cite{Pas3}, \cite{Sen1}, \cite{Sen2}.  We first introduce a formally adjoint differential expression for $\cL$:
\begin{equation*}
\cL^*:= -\sum\limits_{i,j=1}^{d} \frac{\p\ }{\p x_i} \overline{A_{ji}} \frac{\p\ }{\p x_j}
  - \sum\limits_{j=1}^{n}\frac{\p\ }{\p x_j}\overline{A_j}  +\overline{A_0}.
\end{equation*}
In view of the condition $A_j\in W_\infty^1(\Om;\mathds{C}^n)$, this differential expression is of the same structure as $\cL$. In particular, for $u,v\in W_2^2(\Om;\mathds{C}^n)$ with the zero trace on $\p\Om$ we have
$(u,\cL^* v)_{L_2(\Om;\mathds{C}^n)}=\hf^0(u,v)$. Using this identity, we can reproduce the proof of the solvability of problem (\ref{2.11}) and we see that with the same $\l_0$ the problem
\begin{equation}\label{4.7}
(\cL^*-\overline{\l})w=h\quad\text{in}\quad\Om,\qquad w=0\quad\text{on}\quad\p\Om,
\end{equation}
is uniquely solvable in $W_2^2(\Om;\mathds{C}^n)$ as $\RE\l<\l_0$ for each $h\in L_2(\Om;\mathds{C}^n)$ and
\begin{equation}\label{4.8}
\|w\|_{W_2^2(\Om;\mathds{C}^n)}\leqslant C(\l)\|h\|_{L_2(\Om;\mathds{C}^n)},
\end{equation}
where a constant $C(\l)$ is independent of $h$. We choose $h$ in (\ref{4.7}) as $h:=v_\e$ in $\Om^\e$, $h=0$ in $\tht^\e$; the corresponding solution is denoted by $w_\e$. Then we multiply the equation in (\ref{4.7}) by $v_\e$ and integrate once by parts over $\Om^\e$. This gives:
\begin{equation}\label{4.9}
\|v_\e\|_{L_2(\Om;\mathds{C}^n)}^2=\hf(v_\e,w_\e)-\l (v_\e,w_\e)_{L_2(\Om;\mathds{C}^n)}  - \left(v_\e,\frac{\p w_\e}{\p\boldsymbol{\nu}^*}\right)_{L_2(\p\tht^\e;\mathds{C}^n)},
\end{equation}
where
\begin{equation*}
\frac{\p \ }{\p\boldsymbol{\nu}^*}:=\sum\limits_{i,j=1}^{d} \nu_i\overline{A_{ji}} \frac{\p\ }{\p x_j}+\sum\limits_{i=1}^{d} \nu_i\overline{A_i} \frac{\p\ }{\p x_j}.
\end{equation*}
Then we write integral identity (\ref{2.18}) with $w_\e$ as the test function and rewrite identity (\ref{4.13}) with $v_\e$ replaced by $w_\e$. The difference of two obtained relations gives an identity similar to (\ref{4.1}):
\begin{equation*}
\hf(v_\e,w_\e)-\l(v_\e,w_\e)_{L_2(\Om^\e;\mathds{C}^n)}= -\big(a^\e(\,\cdot\,,u_\e),w_\e\big)_{L_2(\p\tht^\e;\mathds{C}^n)}
-\left(\frac{\p u_0}{\p\boldsymbol{\nu}},w_\e\right)_{L_2(\p\tht^\e;\mathds{C}^n)}.
\end{equation*}
It follows from this identity and (\ref{4.9}) that
\begin{equation}\label{4.10}
\|v_\e\|_{L_2(\Om;\mathds{C}^n)}^2=   -\big(a^\e(\,\cdot\,,u_\e),w_\e\big)_{L_2(\p\tht^\e;\mathds{C}^n)}
- \left(v_\e,\frac{\p w_\e}{\p\boldsymbol{\nu}^*}\right)_{L_2(\p\tht^\e;\mathds{C}^n)}- \left(\frac{\p u_0}{\p\boldsymbol{\nu}},w_\e\right)_{L_2(\p\tht^\e;\mathds{C}^n)}.
\end{equation}

Conditions (\ref{2.6}) yield that $|a^\e(x,u_\e)|\leqslant \mu(\e)|u_\e|$.
Using this inequality, (\ref{2.12}), Lemma~\ref{lm3.1} and the identity
\begin{equation*}
\int\limits_{\p\om_k^\e} \frac{\p w_\e}{\p\boldsymbol{\nu}^*}\,ds = -\int\limits_{\om_k^\e}\Bigg(\sum\limits_{i,j=1}^{d} \frac{\p\ }{\p x_i} A_{ij} \frac{\p w_\e}{\p x_j}+\sum\limits_{i=1}^{d} \frac{\p\ }{\p x_i}\overline{A_i} w_\e\Bigg)\,dx
\end{equation*}
instead of (\ref{4.25}),  we estimate as in (\ref{4.16}), (\ref{4.17}), (\ref{4.18}),  (\ref{4.19}):
\begin{equation}\label{4.20}
\begin{aligned}
& \Big|\big(a^\e(\,\cdot\,,u_\e),w_\e\big)_{L_2(\p\tht^\e;\mathds{C}^n)}\Big| \leqslant C \mu\|u_\e\|_{L_2(\p\tht^\e;\mathds{C}^n)} \|w_\e\|_{L_2(\p\tht^\e;\mathds{C}^n)}
\\
&\hphantom{\Big|\big(a^\e(\,\cdot\,,u_\e),w_\e\big)_{L_2(\p\tht^\e;\mathds{C}^n)}\Big| }\leqslant C  \big(\e^2\eta^2\kappa+\eta^{\frac{n}{2}}\vk^{\frac{1}{2}} + \e^{-1}\eta^{n-1}\big)\mu \|f\|_{L_2(\Om;\mathds{C}^n)} \|w_\e\|_{W_2^2(\Om;\mathds{C}^n)},
\\
&\Bigg|\left(v_\e,\frac{\p w_\e}{\p\boldsymbol{\nu}^*}\right)_{L_2(\p\tht^\e;\mathds{C}^n)}\Bigg|
\leqslant C \big(\e\eta\vk^\frac{1}{2}+  \eta^\frac{n}{2}\big) \|v_\e\|_{W_2^1(\Om;\mathds{C}^n)}\|w_\e\|_{W_2^2(\Om;\mathds{C}^n)}.
\end{aligned}
\end{equation}

Let us estimate the third term in the right hand side of (\ref{4.10}). Since the function $u_0$ belongs to $W_2^2(\Om;\mathds{C}^n)$, we can rewrite this term   via the following integration by parts:
\begin{align*}
\bigg(\frac{\p u_0}{\p\boldsymbol{\nu}},w_\e\bigg)_{L_2(\p\tht^\e;\mathds{C}^n)}
=&-((\cL-\l)u_0,w_\e)_{L_2(\tht^\e;\mathds{C}^n)}
 - \sum\limits_{i,j=1}^{d} \bigg(A_{ij} \frac{\p u_0}{\p x_j},\frac{\p w_\e}{\p x_i}\bigg)_{L_2(\tht^\e;\mathds{C}^n)}
 \\
 &+\sum\limits_{j=1}^{d} \bigg(A_j \frac{\p u_0}{\p x_j},  w_\e \bigg)_{L_2(\tht^\e;\mathds{C}^n)} +\big((A_0-\l)u_0,w_\e)_{L_2(\tht^\e;\mathds{C}^n)}.
\end{align*}
By  Lemma~\ref{lm3.6} and the equation in (\ref{2.11}) we then obtain:
\begin{align*}
\bigg|\bigg(\frac{\p u_0}{\p\boldsymbol{\nu}},w_\e\bigg)_{L_2(\p\tht^\e;\mathds{C}^n)} \bigg| \leqslant &C \big(\e\eta\vk^\frac{1}{2}+\eta^\frac{n}{2}\big) \|f\|_{L_2(\tht^\e;\mathds{C}^n)} \|w_\e\|_{W_2^1(\Om;\mathds{C}^n)}
\\
&+ C\big(\e^2\eta^2\vk+\eta^n\big) \|u_0\|_{W_2^2(\Om;\mathds{C}^n)}\|w_\e\|_{W_2^2(\Om;\mathds{C}^n)}.
\end{align*}
This inequality  and (\ref{4.20}), (\ref{4.10}), (\ref{4.8}), (\ref{4.5}) imply
\begin{align*}
\|v_\e\|_{L_2(\Om;\mathds{C}^n)}^2\leqslant &C \big(\e^2\eta^2\kappa+\eta^{\frac{n}{2}}\vk^{\frac{1}{2}} + \e^{-1}\eta^{n-1}\big)\mu \|f\|_{L_2(\Om;\mathds{C}^n)} \|w_\e\|_{W_2^2(\Om;\mathds{C}^n)}
\\
&+ C \big(\e\eta\vk^\frac{1}{2}+  \eta^\frac{n}{2}\big) \|v_\e\|_{W_2^1(\Om;\mathds{C}^n)}\|w_\e\|_{W_2^2(\Om;\mathds{C}^n)}
\\
&+C \big(\e\eta\vk^\frac{1}{2}+\eta^\frac{n}{2}\big) \|f\|_{L_2(\tht^\e;\mathds{C}^n)} \|w_\e\|_{W_2^1(\Om;\mathds{C}^n)}
\\
&+ C\big(\e^2\eta^2\vk+\eta^n\big) \|u_0\|_{W_2^2(\Om;\mathds{C}^n)}\|w_\e\|_{W_2^2(\Om;\mathds{C}^n)}
\\
\leqslant & C\Big(\big(\big(\e^2\eta^2\kappa+\eta^{\frac{n}{2}}\vk^{\frac{1}{2}} + \e^{-1}\eta^{n-1}\big)\mu + \e^2\eta^2\vk+\eta^n \big) \|f\|_{L_2(\Om;\mathds{C}^n)}
\\
&+\big(\e\eta\vk^\frac{1}{2}+\eta^\frac{n}{2}\big) \|f\|_{L_2(\tht^\e;\mathds{C}^n)} \Big)\|v_\e\|_{L_2(\Om^\e;\mathds{C}^n)}.
\end{align*}
The obtained inequality proves (\ref{2.17}).

\section{Sharpness of estimates}\label{sec:sharp}

In this section we study the sharpness of inequality (\ref{2.16}). The main idea is to find appropriate examples of problems (\ref{2.7}), for which we can estimate from below the norms $\|u_\e-u_0\|_{W_2^1(\Om^\e)}$ by the same expressions as in the right hand side of (\ref{2.16}).
Throughout this section, in various estimates, we often use conditions (\ref{2.10}) not saying this explicitly. We also consider only the case of a single scalar equation, that is, as $d=1$.

We let $\cL:=-\D$, $a^\e(x,u)=\mu(\e) u$. One of the points $M_k^\e$ is chosen to be zero and we suppose that all other points $M_k^\e$ are separated from zero by a distance at least $1$. All domains $\om_{k,\e}$ are chosen as $\om_{k,\e}:=B_1(0)$. Under such choice, each cavity is a ball of the radius $\e\eta$. It is clear that  we can take $\l_0=-1$. We then let $R_1=R_2=1$ and suppose that $R_3$ can be taken obeying $R_3\geqslant \frac{4}{3}$.

We first assume that $\mu=0$. Let $U_0=U_0(\xi)$ be an infinitely differentiable function defined in $\mathds{R}^d\setminus B_1(0)$ and vanishing outside $B_{\frac{11}{10}}(0)$. We suppose that $U_0$ is not identically zero in $B_{\frac{11}{10}}(0)\setminus B_1(0)$ and
\begin{equation}\label{4.22}
\frac{\p U_0}{\p|\xi|}=1\quad\text{on}\quad \p B_1(0).
\end{equation}
Then we let $u_0(x):=U_0(x\e^{-1}\eta^{-1})$ and $f_0:=(\cL+1) u_0$. Having the formula
\begin{equation*}
f_0(x)=\e^{-2}\eta^{-2}(-\D_\xi+\e^2\eta^2) U_0(x\e^{-1}\eta^{-1})
\end{equation*}
in mind, it is clear that we can choose the function $U_0$ so that
\begin{equation}\label{4.23}
\|f_0\|_{L_2(\Om)}\geqslant C (\e\eta)^{\frac{n}{2}-2};
\end{equation}
throughout this section by $C$ we denote various inessential fixed positive constant independent of $\e$ and $\eta$. The function $u_0$ obviously solves problem (\ref{2.11}).

The ordinary differential equation
\begin{equation*}
\left(-\frac{1}{r^{n-1}}\frac{d\ }{dr}r^{n-1}\frac{d\ }{dr} + 1\right)X=0,\qquad r>0,
\end{equation*}
has a non-trivial infinitely differentiable solution exponentially decaying at infinity and behaving at zero as
\begin{equation}\label{4.26}
\begin{aligned}
&X(r)=r^{-n+2}+O(r^{-n+2}) && \text{for odd}\quad n\geqslant 3,
\\
&X(r)=r^{-n+2}+O(r^{-n+2}\ln r)\quad && \text{for even}\quad n\geqslant 4,
\\
&X(r)=-\ln r+O(r^2\ln r) && \text{for}\quad n=2.
\end{aligned}
\end{equation}
This solution can be expressed in terms of an appropriate Bessel function.  It is obvious that the function $X(|x|)$ solves the equation $(\cL+1)X(|x|)=0$ in $\mathds{R}^d$.

We then let $u_\e$ to be the solution to problem (\ref{2.7}) with the above chosen function $f$ and we define one more function:
\begin{equation}\label{4.27}
\tilde{u}_\e(x):=u_0(x)-\frac{1}{\e\eta X'(\e\eta)} X(|x|)\chi_1(|x|),
\end{equation}
where $\chi_1=\chi_1(t)$ is an infinitely differentiable cut-off function equalling to one as $|t|<\tfrac{1}{2}$ and vanishing as $|t|>\tfrac{2}{3}$. This function solves problem (\ref{2.7}) with the right hand side
\begin{equation*}
f+\frac{1}{\e\eta X'(\e\eta)}F_{\e,1},\quad\text{where}\quad
F_1(x):=2\nabla X(|x|)\cdot \nabla \chi_1(|x|) + X(|x|)\D\chi_1(|x|).
\end{equation*}
It follows from the definition of the cut-off function $\chi_1$ that the function $F_1$ is non-zero only as $\tfrac{1}{2}<|x|<\tfrac{2}{3}$. Employing then asymptotics (\ref{4.26}), by straightforward calculations we confirm that
\begin{equation}\label{4.30}
\|F_1\|_{L_2(\Om)}\leqslant  C,\qquad
\frac{1}{\e\eta |X'(\e\eta)|}\|F_1\|_{L_2(\Om)} \leqslant C (\e\eta)^{n-2}.
\end{equation}
Apriori estimate (\ref{2.12}) and inequality (\ref{4.23}) then imply the estimate
\begin{equation}\label{4.32}
\|u_\e-\tilde{u}_\e\|_{W_2^1(\Om)}\leqslant C (\e\eta)^{n-2} \leqslant C(\e\eta)^{\frac{n}{2}+2}\|f\|_{L_2(\Om)}.
\end{equation}
At the same time, using the definition of the function $\chi_1$ and identities (\ref{4.26}),
by straightforward calculations we find that
\begin{equation}\label{4.37}
\|\nabla X(|x|)\chi(|x|)\|_{L_2(\Om)}\geqslant C (\e\eta)^{-\frac{n}{2}+1}\vk^\frac{1}{2},
\end{equation}
and therefore, due to (\ref{4.23}),
\begin{equation*}
\|\nabla (u_\e-u_0)\|_{L_2(\Om)}\geqslant C (\e\eta)^{\frac{n}{2}-1} \vk^\frac{1}{2} \geqslant C \e\eta \vk^\frac{1}{2}\|f\|_{L_2(\Om)}.
\end{equation*}
This estimate and (\ref{4.32}) prove that the term $\e\eta\vk^{\frac{1}{2}}$ in inequality (\ref{2.16}) is order sharp.

In order to prove the order sharpness of the terms in (\ref{2.16}) multiplied by $\mu$, we should compare the solutions of problem (\ref{2.7}) with $\mu=0$ and $\mu\ne0$. Here instead of  (\ref{4.22}) we suppose that the function $U_0$ satisfies the boundary condition
\begin{equation*}
\frac{\p U_0}{\p|\xi|}=0,\qquad U_0=1\qquad\text{on}\qquad \p B_1(0).
\end{equation*}
Then the function $u_0(x):=U_0(x\e^{-1}\eta^{-1})$ solves problem (\ref{2.7}) with $\mu=0$. Let $u_\e$ be the solution to the same problem but for $\mu>0$. Instead of (\ref{4.27}), we introduce a function
\begin{equation*}
\tilde{u}_\e(x):=u_0(x)-\frac{\mu}{X'(\e\eta)+\mu X(\e\eta)} X(|x|)\chi_1(|x|).
\end{equation*}
This function also solves problem (\ref{2.7}) with $\mu>0$ and the right hand side in the equation is \begin{equation*}
f+\frac{\mu}{X'(\e\eta)+\mu X(\e\eta)}F_1.
\end{equation*}
Employing then the first estimate from (\ref{4.30}) and (\ref{2.12}), (\ref{4.23}), (\ref{4.26}), we get an inequality similar to (\ref{4.32}):
\begin{equation}\label{4.36}
\|u_\e-\tilde{u}_\e\|_{W_2^1(\Om)}\leqslant C (\e\eta)^{n-1}\mu \leqslant C(\e\eta)^{n+1}\mu\|f\|_{L_2(\Om)}.
\end{equation}
It also follows from (\ref{4.37}), (\ref{4.23}) that
\begin{equation*}
\|\nabla (u_\e-u_0)\|_{L_2(\Om)}\geqslant C (\e\eta)^{\frac{n}{2}} \vk^\frac{1}{2} \mu\geqslant C \e^2\eta^2\vk^\frac{1}{2}\mu\|f\|_{L_2(\Om)}.
\end{equation*}
This estimate and (\ref{4.36}) proves the order sharpness of the term $\e^2\eta^2\mu$ in (\ref{2.16}) for $n\geqslant 5$.

To show the sharpness of the other terms, we consider a different situation. Namely, we assume that $\mu$ is arbitrary and we choose
\begin{equation*}
\Om=\mathds{R}^d,\qquad  \mathds{I}^\e:=\mathds{Z}^n,\qquad M_k^\e:=\e k,\qquad \om_{k,\e}:=B_1(0).
\end{equation*}
The above choice defines a periodic perforation in $\Om$:
\begin{equation*}
\tht^\e=\bigcup\limits_{k\in \mathds{Z}^n} B_{\e\eta}(k),\qquad \Om^\e=\mathds{R}^d\setminus \tht^\e.
\end{equation*}

Let $U_1=U_1(x)$ be a non-zero infinitely differentiable compactly supported function and $u_0(x):=(-\cL+1)U_1(x\b)$, $f(x):=(-\cL+1) u_0(x)$, where $\b$ is some sufficiently large but fixed positive parameter. Then $u_0$ solves problem (\ref{2.11}) with just introduced function $f$.  By $Y_0=Y_0(\z)$ we denote the $(0,1)^n$-periodic solution to the following boundary value problem
\begin{equation}
\label{5.25}
\begin{gathered}
-\D_\z Y_0=c_2\quad\text{in}\quad \mathds{R}^d\setminus \mathds{Z}^n,\qquad c_2:=\mes \mathbb{S}^{n-1},
\\
\begin{aligned}
&Y_0(\z)=\frac{1}{(2-n)|\z-k|^{n-2}}+O(|\z-k|), \quad&& \z\to k,\quad n\geqslant 3, \quad k\in \mathds{Z}^n,
\\
& Y_0(\z)=\ln|\z-k|+O(|\z-k|), &&\z\to k, \quad n=2, \quad k\in \mathds{Z}^n,
\end{aligned}
\end{gathered}
\end{equation}
where $\mes \mathbb{S}^{n-1}$ is the area of the unit sphere in $\mathds{R}^d$. It is easy to see that  this problem satisfies the standard solvability condition; the uniqueness of the solution  is ensured by the order of the error terms in the prescribed asymptotics.

We denote $u_1(x):= U_1(x\b)$ and
\begin{align*}
&\tilde{u}_\e(x):=u_0(x)+ \a_0(\e) Y_0\left(\frac{x}{\e}\right)u_0(x)  + \sum\limits_{j=1}^{n} \a_j(\e) \frac{\p u_0}{\p x_j} \frac{\p Y_0}{\p\xi_j}\left(\frac{x}{\e}\right)-c_2\e^{-2}\a_0(\e) u_1(x),
\\
&
\begin{aligned}
&\a_0(\e):=\frac{\e\eta^{n-1}(\e)\mu(\e)}{1+\frac{\e\eta(\e)\mu(\e)}{n-2}},\quad
&&\a_j(\e):=\frac{1+\frac{\a_0(\e)}{(2-n)\eta^{n-2}(\e)}}{n-1+\e\eta(\e)\mu(\e)}\e\eta^n(\e)
\quad&&\text{as}\quad n\geqslant 3,
\\
&\a_0(\e):=\frac{\e\eta\mu}{1-\e\eta\mu\ln\eta}, &&  \a_j(\e):=\frac{1+\a_0(\e)\ln\eta(\e)}{1+\e\eta(\e)\mu(\e)}\e\eta^2(\e)
 &&\text{as}\quad n=2.
\end{aligned}
\end{align*}

The smoothness of the function $U_1$ as well as the compactness of its support allow us to apply estimate (\ref{2.16}) to the problem for $u_1$ and this immediately imply:
\begin{equation}\label{4.48}
c_2\e^{-2}\a_0\|u_1\|_{W_2^1(\Om^\e)}=O\big(\e^{-1}\eta^{n-1}\mu\b^{1-\frac{n}{2}}\big);
\end{equation}
here and in the next displayed inequalities the  identities $f_1=O(f_2)$ are to be treated as a shorthand notation for a two-sided estimate $C f_2\leqslant f_1\leqslant C^{-1} f_2$ with some constant $C$ we independent of $\e$, $\eta$, $\mu$ and $\b$. Using asymptotics for the function $Y_0$ in (\ref{5.25}) and the properties of $u_0$, by routine straightforward calculations we check the following identities:
\begin{align*}
& \|\a_0\nabla Y_0 u_0\|_{L_2(\Om^\e)}=O \big(\mu\eta^{\frac{n}{2}}\vk^\frac{1}{2}\b^{2-\frac{n}{2}}+\e\eta^{n-1}\mu\b^{3-\frac{n}{2}} (1+\kappa^\frac{1}{2}\eta^{-\frac{n}{2}+2})\big),
\\
&
\bigg\|\sum\limits_{j=1}^{n} \a_j \nabla \frac{\p u_0}{\p x_j} \frac{\p Y_0}{\p\xi_j}
\bigg\|_{L_2(\Om^\e)}=O\big(\eta^{\frac{n}{2}}\b^{3-\frac{n}{2}} + \e\eta^{\frac{n}{2}+1}\vk^\frac{1}{2}\b^{4-\frac{n}{2}}\big).
\end{align*}
These estimates and (\ref{4.48}) imply that by choosing $\b$ large enough but fixed, we can always get the inequality
\begin{equation}\label{4.49}
\bigg\|\a_0  Y_0 u_0 + \sum\limits_{j=1}^{n} \a_j  \frac{\p u_0}{\p x_j} \frac{\p Y_0}{\p\xi_j} - c_2\e^{-2}\a_0  u_1 \bigg\|_{W_2^1(\Om^\e)}\geqslant C \big(\e^{-1}\eta^{n-1}\mu + \mu\eta^{\frac{n}{2}}\vk^\frac{1}{2} + \eta^{\frac{n}{2}}
\big);
\end{equation}
here and till the end of the section the symbol $C$ stands for positive constants independent of $\e$, $\eta$ and $\mu$. Indeed, if one of the three  functions in the right hand of the above inequality is much greater than the two others, then inequality (\ref{4.49}) is obvious. If at least two of them are of the same order, then we just choose a sufficiently large but fixed $\b$ to ensure that the contributions of the functions in the left hand side of (\ref{4.49}) do not cancel out.

It follows from the definition of $u_1$ and $Y_0$ that the function $\tilde{u}_\e$ solves the equation
\begin{gather*}
(\cL+1)\tilde{u}_\e=f_0+F_{\e,2}\quad\text{in}\quad \Om^\e,
 \\
F_{\e,2}:=-\a_0\big(2\nabla Y_0\cdot \nabla u_0 + Y_0\D u_0 \big) + \sum\limits_{j=1}^{n} \a_j \left( \frac{\p Y_0}{\p \xi_j} \frac{\p u_0}{\p x_j} - 2\nabla \frac{\p Y_0}{\p \xi_j} \cdot \nabla \frac{\p u_0}{\p x_j} - \frac{\p Y_0}{\p \xi_j} \D \frac{\p u_0}{\p x_j} \right).
\end{gather*}
It is straightforward to confirm that the function $F_{\e,2}$ can be represented as follows
\begin{equation}\label{4.41}
\begin{gathered}
F_{\e,2}=F_{\e,2,0} + \sum\limits_{i=1}^{n} \frac{\p\ }{\p x_i} F_{\e,2,i},
\\
F_{\e,2,0}:=\a_0 Y_0 \D u_0+\sum\limits_{j=1}^{n} \a_j \frac{\p Y_0}{\p \xi_j}(\D+1) \frac{\p u_0}{\p x_j},
\qquad F_{\e,2,i}:=-2\a_0 Y_0 \frac{\p u_0}{\p x_i} - 2\sum\limits_{j=1}^{n} \a_j  \frac{\p Y_0}{\p\xi_j} \frac{\p u_0}{\p x_j}.
\end{gathered}
\end{equation}
Using the periodicity of $Y_0$ and its asymptotics at $\xi=k$, see (\ref{5.25}), as well as the smoothness of $u_0$, by straightforward calculations we find that
\begin{equation}\label{4.42}
\begin{aligned}
&\|F_{\e,2,0}\|_{L_2(\mathds{R}^d)}+\sum\limits_{i=1}^{n}\|F_{\e,2,i}\|_{L_2(\mathds{R}^d)}\leqslant C  \e\Big(
\big(\eta^{n-1}+\eta^{\frac{n}{2}+1}\kappa^\frac{1}{2}\big)\mu + \eta^{\frac{n}{2}+1}\vk^\frac{1}{2}\Big),
\\
&|F_{\e,2,i}|_{C(\p\om_k^\e)}\leqslant C (\e\eta\vk\mu+\e\eta).
\end{aligned}
\end{equation}
Using asymptotics for $Y_0$ once again, we see that the function $\tilde{u}_\e$ satisfies the boundary conditions
\begin{equation}\label{4.43}
\frac{\p u_0}{\p\boldsymbol{\nu}}+\mu u_0=\phi_\e\quad\text{on}\quad\p\tht^\e,
\qquad \|\phi_\e\|_{C(\p\om_k^\e)}\leqslant C\big(\eta^{n-1}\mu
+\eta^n+\e\eta\big).
\end{equation}

By $\hat{u}_\e$ we denote the solution to the problem
\begin{equation*}
(\cL+1)\hat{u}_\e=F_{\e,2}\quad\text{in}\quad\Om^\e,\qquad \left(\frac{\p\ }{\p\boldsymbol{\nu}}+\mu\right)\hat{u}_\e=\phi_\e\quad\text{on}\quad\p\tht^\e.
\end{equation*}
This problem can be easily reduced to one with a homogeneous boundary condition, which is obviously solvable. The integral identity associated with this problem with $\hat{u}_\e$ taken as a test function reads as
\begin{equation*}
\|\hat{u}_\e\|_{W_2^1(\Om^\e)} + \mu \|\hat{u}_\e\|_{L_2(\p\tht^\e)}^2=(\phi_\e,\hat{u}_\e)_{L_2(\p\tht^\e)} + (F_{\e,2},\hat{u}_\e)_{L_2(\Om^\e)}.
\end{equation*}
Employing the representation for $F_{\e,2}$ in (\ref{4.41}), we integrate by parts and rewrite the above identity as follows:
\begin{equation*}
\|\hat{u}_\e\|_{W_2^1(\Om^\e)}^2 + \mu \|\hat{u}_\e\|_{L_2(\p\tht^\e)}^2=(\phi_\e,\hat{u}_\e)_{L_2(\p\tht^\e)}+ (F_{\e,2,0},\hat{u}_\e)_{L_2(\Om^\e)} + \sum\limits_{i=1}^{n} (F_{\e,2,i}\nu_i,\hat{u}_\e)_{L_2(\p\tht^\e)}
-\sum\limits_{i=1}^{n}\left(F_{\e,2,i},\hat{u}_\e\right)_{L_2(\Om^\e)}.
\end{equation*}
Owing to estimates in (\ref{4.42}), (\ref{4.43}) and Lemma~\ref{lm3.1}, the right hand side of this identity satisfies the estimate
\begin{align*}
\Big|(&\phi_\e,\hat{u}_\e)_{L_2(\p\tht^\e)}+ (F_{\e,2,0},\hat{u}_\e)_{L_2(\Om^\e)} + \sum\limits_{i=1}^{n} (F_{\e,2,i}\nu_i,\hat{u}_\e)_{L_2(\p\tht^\e)}
-\sum\limits_{i=1}^{n}\left(F_{\e,2,i},\hat{u}_\e\right)_{L_2(\Om^\e)}\Big|
\\
\leqslant & C\e^{\frac{1}{2}}\eta^{\frac{n-1}{2}}\big(\e^{\frac{1}{2}}\eta^{\frac{1}{2}}\vk^{\frac{1}{2}} + \e^{-\frac{1}{2}}\eta^{\frac{n-1}{2}}\big)\big( (\eta^{n-1}+\eta^{\frac{n}{2}+1}\kappa^\frac{1}{2})\mu+ \e\eta^{\frac{n}{2}+1}\vk^\frac{1}{2}\big)
\big(\e\eta\vk\mu+\e\eta+\eta^{n-1}\mu
+\eta^n\big)
 \|\hat{u}_\e\|_{W_2^1(\Om^\e)}
 \\
\leqslant & C \big(\e\eta^{\frac{1}{2}}\vk^{\frac{1}{2}} + \eta^{\frac{n-1}{2}}\big) \big( (\eta^{\frac{3(n-1)}{2}}+\eta^{n+\frac{1}{2}}\kappa^\frac{1}{2})\mu+ \eta^{n+\frac{1}{2}}\vk^\frac{1}{2}\big)
\big(\e\eta\vk\mu+\e\eta+\eta^{n-1}\mu
+\eta^n\big) \|\hat{u}_\e\|_{W_2^1(\Om^\e)}.
 \end{align*}
Due to conditions (\ref{2.10}), the first and the third brackets in the right hand side of the above inequality tend to zero as $\e\to+0$, while the second bracket satisfy the relations:
\begin{align*}
\big((\eta^{\frac{3(n-1)}{2}}+\eta^{n+\frac{1}{2}}\kappa^\frac{1}{2})\mu+ \eta^{n+\frac{1}{2}}\vk^\frac{1}{2}\big)=&\e\eta^{\frac{n-1}{2}}\e^{-1}\eta^{n-1}\mu + \eta^\frac{n+1}{2}\kappa^{\frac{1}{2}} \mu\eta^\frac{n}{2} + \eta^{\frac{n+1}{2}}\eta^{\frac{n}{2}} \vk^\frac{1}{2}
\\
\leqslant &\max\Big\{\e\eta^{\frac{n-1}{2}},\eta^\frac{n+1}{2}\kappa^{\frac{1}{2}}, \eta^{\frac{n+1}{2}}\Big\} \big(
\e^{-1}\eta^{n-1}\mu + \mu\eta^\frac{n}{2} + \eta^{\frac{n}{2}} \vk^\frac{1}{2}\big).
\end{align*}
Hence, in view of estimate  (\ref{4.49}) we conclude that the terms $\e^{-1}\eta^{n-1}\mu$, $\eta^\frac{n}{2}\vk^\frac{1}{2}\mu$ and $\eta^\frac{n}{2}$ in (\ref{2.16}) are order sharp.

\section*{Funding}

The work is supported by  the Czech Science Foundation within the project 22-18739S.

\section*{Conflict of interest}

The author declares that he has no conflicts of interest.

\section*{Data availability}

Not applicable in the manuscript as no datasets were generated or analysed during the current study.


\begin{thebibliography}{99}


\bibitem{Post}  Ann\'e, C.,  Post, O.: Wildly perturbed manifolds:
norm resolvent and spectral convergence. J. Spectr. Theory {\bf 11}:1, 229--279 (2021).

\bibitem{BK} Borisov, D.I., K\v{r}\'{\i}\v{z}, J.: Operator estimates for non-periodically perforated domains with Dirichlet and nonlinear Robin conditions: vanishing limit. Preprint: arXiv 2204.04829 (2022).

\bibitem{MSB21} Borisov, D.I., Mukhametrakhimova, A.I.: Uniform convergence and asymptotics for problems in domains finely perforated along a prescribed manifold in the case of the homogenized Dirichlet condition.  Sb. Math. {\bf 212}:8, 1068--1121 (2021).

\bibitem{AA22} Borisov, D.I., Mukhametrakhimova, A.I.: Norm convergence for problems with perforation along a given manifold
  with nonlinear Robin condition on boundaries of cavities. Preprint: arXiv:2202.10767 (2022).

\bibitem{PMA22-2} Borisov, D.I.: On operator estimates for planar domains with non-regular curving of boundary: Dirichlet and Neumann condition. J. Math. Sci. {\bf 254}:5, 562--580 (2022).


\bibitem{PRSE16} Borisov, D., Cardone, G.,  Durante, T.: Homogenization and uniform resolvent convergence for elliptic operators in a strip perforated along a curve.  Proc. Roy. Soc. Edinburgh. Sec. A. Math. {\bf 146}:6, 1115--1158 (2016).


\bibitem{JST22} Borisov, D.:  Operator estimates for non-periodically perforated domains with Dirichlet and nonlinear Robin conditions: strange term. Preprint: arXiv:2205.09490, 2022.


\bibitem{ChDR} Cherednichenko K., Dondl P., and R\"osler F.: Norm-resolvent convergence in perforated
domains. Asymp. Anal. {\bf 110}:3-4, 163--184 (2018).

\bibitem{Sha} D\'{\i}az, J.I., G\'omez-Castro, D., Shaposhnikova, T.A.:
Nonlinear reaction-diffusion processes for nanocomposites: anomalous improved homogenization. De Gruyter, Berlin, 2021.


\bibitem{Dub} Dubinskii, Yu.A.: Nonlinear elliptic and parabolic equations.  J. Math. Sci., {\bf 12}:5, 475--554 (1979).



\bibitem{ZKO}   Jikov, V.V., Kozlov, S.M., Oleinik, O.A.:
Homogenization of differential operators. Fiziko-Matematicheskaya Literatura, Moscow  (1993). (in Russian).


\bibitem{Khra} Khrabustovskyi, A., Plum, M.: Operator estimates for homogenization of the Robin Laplacian in a perforated domain. J. Diff. Equats.  {\bf 338}, 474--517 (2022).


\bibitem{Khra2}
  Khrabustovskyi, A., Post. O.: Operator estimates for the crushed ice problem. Asymp. Anal. {\bf 110}(3-4), 137--161 (2018).



\bibitem{MK} Marchenko, V.A., Khruslov, E.Ya.:
Boundary value problems in domains with a fine-grained boundary, Naukova Dumka, Kiev, (1974). (in Russian).

\bibitem{MK1} Marchenko, V.A., Khruslov, E.Ya.:   Homogenization of
partial differential equations. Birkh\"auser, Boston (2006).


\bibitem{OIS} Oleinik, O.A., Iosifyan, G.A., Shamaev, A.S.:
Mathematical problems in elasticity and homogenization.   North-Holland, Amsterdam  (1992).

\bibitem{Past}   Pastukhova, S.E.: Resolvent approximations in $L_2$-norm for elliptic operators acting in a perforated space. Contem. Math. Fund. Direct. {\bf 66}:2, 314--334 (2020). (in Russian).


 \bibitem{Pas1}  Pastukhova S.E.: Homogenization estimates for singularly
perturbed operators. J. Math. Sci. {\bf 251}:5, 724--747 (2020).


\bibitem{Pas2}
 Pastukhova S.E: $L_2$-approximation of resolvents in
homogenization of higher order
elliptic operators. J. Math. Sci. {\bf 251}:6, 902--925 (2020).


 \bibitem{Pas3}
Pastukhova S.E.: $L_2$-approximation of resolvents in
homogenization of  fourth-order elliptic operators. Sb. Math. {\bf 212}:1, 111--134 (2020).


 \bibitem{Sen1} Senik N.N.: Homogenization for non-self-adjoint periodic elliptic operators on an infinite cylinder. SIAM J. Math. Anal. 49:2, 874--898 (2017).

 \bibitem{Sen2} Senik N.N.: Homogenization for locally periodic elliptic operators. J. Math. Anal. Appl. {\bf 505}:2, id 125581 (2021).


\bibitem{Sus} Suslina, T.A.: Spectral approach to homogenization of elliptic operators in a perforated space. Rev. Math. Phys. {\bf 30}:08, 1840016 (2018).

\bibitem{Vain} Vainberg, M.M.:
Variational method and method of monotone operators in the theory of nonlinear equations.
A Halsted Press Book, New York-Toronto; John Wiley \& Sons; Jerusalem-London. (1973).

\bibitem{Zhi} Zhikov, V.V.: Spectral method in homogenization theory. Proc. Steklov Inst. Math. {\bf 250}, 85--94 (2005).


\end{thebibliography}
\end{document}